\documentclass[a4paper,10pt]{article}
\usepackage[latin1]{inputenc}
\usepackage[T1]{fontenc}
\usepackage{amsmath}
\usepackage{amsfonts}
\usepackage{amstext}
\usepackage{amsthm}
\usepackage[all]{xy}
\usepackage{geometry}
\usepackage{srcltx}
\usepackage{graphics}
\usepackage{epsfig}
\usepackage{subfigure}
\usepackage{slashed}
\usepackage{color}
\usepackage{amssymb}
\usepackage{srcltx}
\newtheorem{theorem}{Theorem}[section]
\newtheorem{lemma}[theorem]{Lemma}

\newtheorem{rem}[theorem]{Remark}

\newtheorem{cor}[theorem]{Corollary}

\DeclareMathOperator{\im}{im}

\DeclareMathOperator{\sfl}{sf}

\DeclareMathOperator{\sgn}{sgn}

\DeclareMathOperator{\gra}{graph}
\DeclareMathOperator{\codim}{codim}

\title{On a Comparison Principle and the Uniqueness of Spectral Flow}
 
\author{Maciej Starostka and Nils Waterstraat}

\begin{document}
\date{}
\maketitle

\footnotetext[1]{{\bf 2010 Mathematics Subject Classification: Primary 58J30; Secondary 37J45, 53D12 }}
\footnotetext[2]{The authors were partly supported by Grant Beethoven2 of the National Science Centre, Poland, no. 2016/23/G/ST1/04081 and DFG Grant AB 360/3-1}

\begin{abstract}
\noindent
The spectral flow is a well-known quantity in spectral theory that measures the variation of spectra about $0$ along paths of selfadjoint Fredholm operators. The aim of this work is twofold. Firstly, we consider homotopy invariance properties of the spectral flow and establish a simple formula which comprises its classical homotopy invariance and yields a comparison theorem for the spectral flow under compact perturbations. We apply our result to the existence of non-trivial solutions of boundary value problems of Hamiltonian systems. Secondly, the spectral flow was axiomatically characterised by Lesch, and by Ciriza, Fitzpatrick and Pejsachowicz under the assumption that the endpoints of the paths of selfadjoint Fredholm operators are invertible. We propose a different approach to the uniqueness of spectral flow which lifts this additional assumption. As application of the latter result, we discuss the relation between the spectral flow and the Maslov index in symplectic Hilbert spaces.    
\end{abstract}

\section{Introduction}
The spectral flow is a homotopy invariant for paths of selfadjoint Fredholm operators that was invented by Atiyah, Patodi and Singer in their famous study of spectral asymmetry and index theory in \cite{APS}. Selfadjoint Fredholm operators are either invertible or $0$ is an isolated eigenvalue of finite multiplicity. Roughly speaking, if $\mathcal{A}=\{\mathcal{A}_\lambda\}_{\lambda\in[0,1]}$ is a path of selfadjoint Fredholm operators, then the spectral flow of $\mathcal{A}$ is the net number of eigenvalues of $\mathcal{A}_0$ that become positive whilst the parameter $\lambda$ travels along the unit interval. In this paper, we are dealing with paths of (generally) unbounded operators which are continuous with respect to the gap-metric. The spectral flow was introduced in this setting by Boo{\ss}-Bavnbek, Lesch and Phillips in \cite{UnbSpecFlow}. Every norm-continuous path of bounded operators is continuous with respect to the gap-metric, and there are various works which had considered this case previously (see, e.g., \cite{SFLPejsachowicz}, \cite{Phillips}).\\
As indicated by the title, this paper falls naturally into two parts. The homotopy invariance of the spectral flow has been stated in various forms. For example, a common property of the spectral flow is that it is invariant under homotopies of paths having invertible endpoints (see e.g. \cite{SFLPejsachowicz}). A more general observation is the invariance under homotopies where the dimensions of the kernels of the endpoints are constant (see e.g. \cite{UnbSpecFlow}). We obtain a general formula for the change of the spectral flow under gap-continuous homotopies of selfadjoint Fredholm operators which comprises all previously known results. As a corollary, we show that the spectral flow is invariant under free homotopies of closed paths which was noted for paths of bounded operators in \cite{SFLPejsachowicz}. After this degression about homotopies, we focus on our first aim of this paper and consider relatively compact perturbations of paths of selfadjoint Fredholm operators. Our main theorem in this part is a \textit{comparison principle} for the spectral flow of paths of relatively compact perturbations of a given path of gap-continuous selfadjoint Fredholm operators. We pay particular attention to the gap-continuity of the perturbed path, which requires a generalisation of a well-known theorem from \cite{Kato}. The comparison principle finally follows from our previous investigation of the homotopy invariance and a method for computing the spectral flow from \cite{WaterstraatHomoclinics} that is based on previous work by Robbin and Salamon from \cite{Robbin}. As an application of the comparison principle, we consider paths of boundary value problems for linear Hamiltonian systems and obtain an estimate for the number of parameter values of the path where the Hamiltonian systems have non-trivial solutions.\\
The spectral flow is uniquely characterised by some of its properties, which is known as the \textit{uniqueness of spectral flow}. For paths of bounded operators, this was independently observed by Ciriza, Fitzpatrick and Pejsachowicz in \cite{JacoboUniqueness} and by Lesch in \cite{Lesch}. Lesch also showed that the same result is true in the case of paths of unbounded operators that are continuous in the gap-topology. However, all these theorems assume that the paths of operators have invertible endpoints. The second aim of this paper is to establish a \textit{uniqueness of spectral flow theorem} which lifts the assumption on the invertibility of the endpoints. As an application of the latter result we consider symplectic Hilbert spaces as in \cite{Furutani} and recall that the graphs of gap-continuous paths of selfadjoint Fredholm operators yield paths of Lagrangian subspaces in this setting. We show that the spectral flow is the Maslov index of the path of graphs as a rather simple consequence of our uniqueness theorem. This fact might be considered as folklore, but we are not aware of a proof in the literature.\\
Our paper is structured as follows. We recall the definition of the spectral flow in the next section. In Section 3, we firstly discuss the homotopy invariance. Afterwards we show an estimate that allows to prove the continuity of paths of unbounded operators in broad generality. The following main theorem of this part of the paper is a comparison result for the spectral flow that we then apply to paths of boundary value problems of Hamiltonian systems. The fourth section of our paper is devoted to the uniqueness of the spectral flow. We firstly recall Lesch's Uniqueness Theorem from \cite{Lesch}, where we in particular introduce the axioms which uniquely characterise the spectral flow. Afterwards we state and prove our main theorem, which is a uniqueness theorem for the spectral flow that, in contrast to \cite{Lesch}, does not require the invertibility of the endpoints of the paths. Finally, we consider symplectic Hilbert spaces and show that the spectral flow of a path can be obtained as Maslov index of its associated path of graphs.


\section{The Spectral Flow}

\subsection{Definition and First Properties}\label{section-sfldef}
The aim of this section is to recall the construction of the spectral flow and some of its basic properties, where we follow \cite{Phillips} and \cite{UnbSpecFlow}.\\
Let $H$ be a real or complex separable Hilbert space. We denote by $\mathcal{C}(H)$ the set of all densely defined closed operators on $H$. The \textit{gap-metric} on $\mathcal{C}(H)$ is defined by

\begin{align}\label{gap}
d_G(T,S)=\|P_T-P_S\|,\quad T,S\in\mathcal{C}(H),
\end{align}
where $P_T, P_S$ are the orthogonal projections onto the graphs of $T$ and $S$. As every selfadjoint operator is closed, these operators form a subset of $\mathcal{C}(H)$ which we denote by $\mathcal{C}^\textup{sa}(H)$. Further, we let $\mathcal{CF}(H)\subset\mathcal{C}(H)$ be the set of all Fredholm operators and we set $\mathcal{CF}^\textup{sa}(H)=\mathcal{C}^\textup{sa}(H)\cap\mathcal{CF}(H)$ which is the set of all selfadjoint Fredholm operators. Let us mention for later reference that, when restricted to the bounded operators $\mathcal{B}(H)\subset\mathcal{C}(H)$, the gap-metric induces the same topology as the metric induced by the operator norm (see \cite[Rem. IV.2.16]{Kato}).\\ 
The spectrum of $T\in\mathcal{CF}^\textup{sa}(H)$ is the (generally non-disjoint) union of the point spectrum and the essential spectrum. Moreover, $0$ is either in the resolvent set or it is an isolated eigenvalue of finite multiplicity (see, e.g., \cite[Lemma 13]{Fredholm}). We denote for $a,b\notin\sigma(T)$ by $\chi_{[a,b]}(T)$ the spectral projection of $T$ with respect to the interval $[a,b]$. Note that if $\sigma_{ess}(T)\cap{[a,b]}=\emptyset$, then $\chi_{[a,b]}(T)$ is the orthogonal projection onto the direct sum of the eigenspaces for eigenvalues in $[a,b]$. The proof of the following lemma can be found in \cite[Lemma 2.9]{UnbSpecFlow}.

\begin{lemma}\label{lemma-sflconstruction}
Let $T_0\in\mathcal{CF}^\textup{sa}(H)$ and $a>0$ such that $\pm a\notin\sigma(T_0)$ and $[-a,a]\cap\sigma_{ess}(T_0)=\emptyset$. Then there is an open neighbourhood $N_{T_0,a}$ of $T_0$ in $\mathcal{CF}^\textup{sa}(H)$ such that $\pm a\notin\sigma(T)$, $[-a,a]\cap\sigma_{ess}(T)=\emptyset$ for all $T\in N_{T_0,a}$ and

\[N_{T_0,a}\ni T\mapsto\chi_{[-a,a]}(T)\in\mathcal{B}(H)\]
is continuous.
\end{lemma}
\noindent
The construction of the spectral flow of a path $\mathcal{A}=\{\mathcal{A}_\lambda\}_{\lambda\in I}$ is now as follows, where we denote by $I$ the unit interval. For every $\lambda\in I$ there is an open neighbourhood $N_{\lambda,a}\subset\mathcal{CF}^\textup{sa}(H)$ of $\mathcal{A}_\lambda$ as in the previous lemma. The preimages of these neighbourhoods define an open covering of the compact interval $I$. Consequently, there is a partition of the interval $0=\lambda_0<\lambda_1<\ldots<\lambda_N=1$ and numbers $a_i>0$ such that 

\[\pm a_i\notin\sigma(\mathcal{A}_\lambda),\quad [-a_i,a_i]\cap\sigma_{ess}(\mathcal{A}_\lambda)=\emptyset,\qquad \lambda\in[\lambda_{i-1},\lambda_i]\] 
as well as

\[[\lambda_{i-1},\lambda_i]\ni\lambda\mapsto\chi_{[-a_i,a_i]}(\mathcal{A}_\lambda)\in\mathcal{B}(H)\]
is continuous, $i=1,\ldots,N$. The \textit{Spectral Flow} of the path $\mathcal{A}$ is the integer

\begin{align}\label{def-sfl}
\sfl(\mathcal{A})=\sum^N_{i=1}{\left(\dim(\im(\chi_{[0,a_i]}(\mathcal{A}_{\lambda_i}))-\dim(\im(\chi_{[0,a_i]}(\mathcal{A}_{\lambda_{i-1}}))\right)}.
\end{align}
A careful analysis of the continuity of the spectral projections in Lemma \ref{lemma-sflconstruction} shows that this definition neither depends on the partition $0=\lambda_0<\ldots<\lambda_N=1$ nor on the numbers $a_1,\ldots,a_N$, and so it is indeed a well-defined index of the path $\mathcal{A}$ (see \cite{Phillips}, \cite{UnbSpecFlow}).\\
In what follows we denote by $\mathcal{A}^1\ast\mathcal{A}^2$ the concatenation of two paths $\mathcal{A}^1, \mathcal{A}^2$ in $\mathcal{CF}^\textup{sa}(H)$ which is defined if $\mathcal{A}^1_1=\mathcal{A}^2_0$, i.e. if the initial point of $\mathcal{A}^2$ is the endpoint of $\mathcal{A}^1$. We note the following property of the spectral flow, which is an immediate consequence of its definition.

\begin{enumerate}
 \item[(C)] If $\mathcal{A}^1$ and $\mathcal{A}^2$ are two paths in $\mathcal{CF}^\textup{sa}(H)$ such that $\mathcal{A}^1_1=\mathcal{A}^2_0$, then 
 \[\sfl(\mathcal{A}^1\ast\mathcal{A}^2)=\sfl(\mathcal{A}^1)+\sfl(\mathcal{A}^2).\]
\end{enumerate}
A further property of the spectral flow that has often been used in the literature is that it vanishes for paths of invertible operators. This is clear from its interpretation and also not difficult to prove by using the continuity of the spectral projections in Lemma \ref{lemma-sflconstruction}. Here we show a slightly more general version of this assertion which will be important in later sections.

\begin{lemma}\label{lemma-kernelconstant}
If $\mathcal{A}:I\rightarrow\mathcal{CF}^\textup{sa}(H)$ is such that $\dim\ker(\mathcal{A}_\lambda)$ is constant for all $\lambda\in I$, then

\[\sfl(\mathcal{A})=0.\]
\end{lemma}

\begin{proof}
Let $\lambda_0\in I$. As $0$ is an isolated eigenvalue of $\mathcal{A}_{\lambda_0}$, there is $\varepsilon>0$ such that the rank of the spectral projection $\chi_{[-\varepsilon, \varepsilon]}(\mathcal{A}_{\lambda_0})$ is the dimension of the kernel of $\mathcal{A}_{\lambda_0}$. As projections of norm-distance less than one have equal ranks (see \cite[Lem. II.4.3]{Gohberg}), it follows from Lemma \ref{lemma-sflconstruction} that there is $\delta>0$ such that 

\[\dim\im(\chi_{[-\varepsilon, \varepsilon]}(\mathcal{A}_{\lambda}))=\dim\im(\chi_{[-\varepsilon, \varepsilon]}(\mathcal{A}_{\lambda_0}))=\dim\ker(\mathcal{A}_{\lambda_0}),\quad \lambda_0-\delta\leq \lambda\leq \lambda_0+\delta.\]
Now $\dim\ker(\mathcal{A}_\lambda)=\dim\ker(\mathcal{A}_{\lambda_0})$ for all $\lambda$ by assumption, and so 

\[\dim\im(\chi_{[-\varepsilon, \varepsilon]}(\mathcal{A}_{\lambda}))=\dim\ker(\mathcal{A}_\lambda),\quad \lambda_0-\delta\leq \lambda\leq \lambda_0+\delta.\]
As $\ker(\mathcal{A}_\lambda)\subset\im(\chi_{[-\varepsilon, \varepsilon]}(\mathcal{A}_{\lambda}))$, this shows that $\im(\chi_{[-\varepsilon, \varepsilon]}(\mathcal{A}_{\lambda}))=\ker(\mathcal{A}_\lambda)$ for $\lambda_0-\delta\leq \lambda\leq \lambda_0+\delta$. Hence we obtain from the definition of the spectral flow that 

\begin{align*}
\sfl(\mathcal{A}\mid_{[\lambda_0-\delta,\lambda_0+\delta]})&=\dim\im(\chi_{[0, \varepsilon]}(\mathcal{A}_{\lambda_0+\delta}))-\dim\im(\chi_{[0, \varepsilon]}(\mathcal{A}_{\lambda_0-\delta}))\\
&=\dim\ker(\mathcal{A}_{\lambda_0+\delta})-\dim\ker(\mathcal{A}_{\lambda_0-\delta})=0,
\end{align*}  
where we have used that $\im(\chi_{[0, \varepsilon]}(\mathcal{A}_{\lambda}))$ is the direct sum of the eigenspaces of $\mathcal{A}_\lambda$ for eigenvalues in $[0,\varepsilon]$, $\lambda\in[\lambda_0-\delta,\lambda_0+\delta]$.\\
Now the assertion follows from the concatenation property (C).
\end{proof}
\noindent
The previous lemma will only be needed in its full generality for discussing the homotopy invariance in the next section. For the uniqueness of the spectral flow, we will use instead the following weaker statement:

\begin{enumerate}
 \item[(Z)] If $\mathcal{A}=\{\mathcal{A}_\lambda\}_{\lambda\in I}$ is a path in $\mathcal{CF}^\textup{sa}(H)$ such that $\mathcal{A}_\lambda$ is invertible for all $\lambda\in I$, then $\sfl(\mathcal{A})=0$.
\end{enumerate}


\section{A Comparison Theorem for the Spectral Flow}

\subsection{A Review of the Homotopy Invariance}\label{section-hominv}
Before we begin our discussion of the homotopy invariance, we note the following important though elementary fact about the spectral flow (see \cite{Phillips}).

\begin{lemma}\label{lemma-homotopytechnical}
Let $a>0$ and $N\subset\mathcal{CF}^\textup{sa}(H)$ be an open set such that $\pm a\notin\sigma(T)$, $[-a,a]\cap\sigma_{ess}(T)=\emptyset$ for all $T\in N$ and such that

\[N\ni T\mapsto \chi_{[-a,a]}(T)\in\mathcal{B}(H)\]
is continuous. If $\mathcal{A}^1$ and $\mathcal{A}^2$ are two paths in $N$ having the same initial and endpoint, i.e. $\mathcal{A}^1_0=\mathcal{A}^2_0$ and $\mathcal{A}^1_1=\mathcal{A}^2_1$, then

\[\sfl(\mathcal{A}^1)=\sfl(\mathcal{A}^2).\] 
\end{lemma}

\begin{proof}
By the definition of the spectral flow, we see that

\begin{align*}
\sfl(\mathcal{A}^1)&=\dim(\im(\chi_{[0,a]}(\mathcal{A}^1_1)))-\dim(\im(\chi_{[0,a]}(\mathcal{A}^1_0)))\\
&=\dim(\im(\chi_{[0,a]}(\mathcal{A}^2_1)))-\dim(\im(\chi_{[0,a]}(\mathcal{A}^2_0)))= \sfl(\mathcal{A}^2),
\end{align*}
where we have used that the initial and endpoints of the paths coincide.
\end{proof}
\noindent
Let us point out that the proof of the following theorem closely follows the proof of the homotopy invariance property in \cite{Phillips}.
 
\begin{theorem}\label{thm-hominv}
Let $h:I\times I\rightarrow\mathcal{CF}^\textup{sa}(H)$ be a homotopy of selfadjoint Fredholm operators. Then

\begin{align}\label{hominv}
\sfl(h(0,\cdot))=\sfl(h(\cdot,0))+\sfl(h(1,\cdot))-\sfl(h(\cdot,1)).
\end{align}
\end{theorem}

\begin{proof}
As $h(I\times I)\subset\mathcal{CF}^\textup{sa}(H)$ is compact, we can find by Lemma \ref{lemma-sflconstruction} an open covering of this set by finitely many open sets $N_i$, $i=1,\ldots,n$, which are as in Lemma \ref{lemma-homotopytechnical}. Now let $\varepsilon$ be a Lebesgue number of the open covering of $I\times I$ made by the $n$ preimages $h^{-1}(N_i)$, i.e. each subset of $I\times I$ of diameter less than $\varepsilon$ is contained in one of the $h^{-1}(N_i)$.\\
We choose a partition $0=\lambda_0\leq\ldots\leq\lambda_m=1$ such that $|\lambda_i-\lambda_{i-1}|\leq\frac{\varepsilon}{\sqrt{2}}$ for $1\leq i\leq m$. Then for each $1\leq i,j\leq m$, the image $h([\lambda_{i-1},\lambda_i]\times[\lambda_{j-1},\lambda_j])$ is contained in one of the sets $N_k$. Let us now consider the four paths obtained from the boundary of the square $[\lambda_{i-1},\lambda_i]\times[\lambda_{j-1},\lambda_j]$, i.e. the two horizontal paths

\begin{align*}
h^h_{i-1,j}(\lambda)=h(\lambda_{i-1},\lambda),\, \lambda\in[\lambda_{j-1},\lambda_j],\qquad h^h_{i,j}(\lambda)=h(\lambda_{i},\lambda),\, \lambda\in[\lambda_{j-1},\lambda_j]
\end{align*}   
and the two vertical paths

\begin{align*}
h^v_{i,j-1}(\lambda)=h(\lambda,\lambda_{j-1}),\, \lambda\in[\lambda_{i-1},\lambda_i],\qquad  h^v_{i,j}(\lambda)=h(\lambda,\lambda_{j}),\, \lambda\in[\lambda_{i-1},\lambda_i].
\end{align*}
If we denote by $(h^v_{i,j})'$ the reverse path of $h^v_{i,j}$, then it readily follows from the definition of the spectral flow that $\sfl((h^v_{i,j})')=-\sfl(h^v_{i,j})$. Moreover, $h^v_{i,j-1}\ast h^h_{i,j}\ast (h^v_{i,j})'$ is a path in $N_k$ having the same initial and endpoint as $h^h_{i-1,j}$. Consequently, by Lemma \ref{lemma-homotopytechnical} and (C), 

\begin{align}\label{equ-proofhomotopy}
\begin{split}
\sfl(h^h_{i-1,j})&=\sfl(h^v_{i,j-1}\ast h^h_{i,j}\ast (h^v_{i,j})')=\sfl(h^v_{i,j-1})+\sfl(h^h_{i,j})+\sfl((h^v_{i,j})')\\
&=\sfl(h^v_{i,j-1})+\sfl(h^h_{i,j})-\sfl(h^v_{i,j}).
\end{split}
\end{align} 
Now,

\begin{align*}
\sfl(h(0,\cdot))&=\sum^m_{j=1}{\sfl(h^h_{0,j})}=\sum^m_{j=1}{\left(\sfl(h^v_{1,j-1})+\sfl(h^h_{1,j})-\sfl(h^v_{1,j})\right)}\\
&=\sfl(h^v_{1,0})-\sfl(h^v_{1,m})+\sum^m_{j=1}{\sfl(h^h_{1,j})}.
\end{align*}
As, again by \eqref{equ-proofhomotopy},

\begin{align*}
\sum^m_{j=1}{\sfl(h^h_{1,j})}&=\sum^m_{j=1}{\left(\sfl(h^v_{2,j-1})+\sfl(h^h_{2,j})-\sfl(h^v_{2,j})\right)}=\sfl(h^v_{2,0})-\sfl(h^v_{2,m})+\sum^m_{j=1}{\sfl(h^h_{2,j})},
\end{align*}
we obtain

\begin{align*}
\sfl(h(0,\cdot))&=\sfl(h^v_{1,0})+\sfl(h^v_{2,0})-\sfl(h^v_{1,m})-\sfl(h^v_{2,m})+\sum^m_{j=1}{\sfl(h^h_{2,j})}.
\end{align*}
If we continue this procedure until we arrive at $i=m$, we get

\begin{align*}
\sfl(h(0,\cdot))&=\sum^m_{i=1}{\sfl(h^v_{i,0})}-\sum^m_{i=1}{\sfl(h^v_{i,m})}+\sum^m_{j=1}{\sfl(h^h_{m,j})}\\
&=\sfl(h(\cdot,0))-\sfl(h(\cdot,1))+\sfl(h(1,\cdot)),
\end{align*}
where we have used (C) in the last equality. This is the claimed equation.
\end{proof}
\noindent
\begin{rem}
There is an alternative way to obtain Theorem \ref{thm-hominv}. One can show at first the homotopy invariance for closed paths, which we below obtain as a corollary of \eqref{hominv}. Then, given a general homotopy $h:I\times I\rightarrow\mathcal{CF}^\textup{sa}(H)$, the spectral flow of the path obtained by restricting $h$ to the boundary of $I\times I$ vanishes. Now \eqref{hominv} readily follows from (C), (Z) and the fact that the spectral flow changes its sign if we reverse the orientation of a path. However, the homotopy invariance for closed paths is hardly easier to prove than \eqref{hominv}.
\end{rem}
\noindent
The following corollary is an immediate consequence of Theorem \ref{thm-hominv} and (Z). It can be considered as the most general form of the homotopy invariance property of the spectral flow.

\begin{cor}\label{cor-hominv}
Let $h:I\times I\rightarrow\mathcal{CF}^\textup{sa}(H)$ be a homotopy of selfadjoint Fredholm operators such that $\sfl(h(\cdot,0))=\sfl(h(\cdot,1))$. Then

\[\sfl(h(0,\cdot))=\sfl(h(1,\cdot)).\]
\end{cor}
\noindent
As a first consequence of Corollary \ref{cor-hominv}, we obtain from Lemma \ref{lemma-kernelconstant} the homotopy invariance property stated in \cite{UnbSpecFlow}.

\begin{cor}\label{cor-hominvkernel}
Let $h:I\times I\rightarrow\mathcal{CF}^\textup{sa}(H)$ be a homotopy of selfadjoint Fredholm operators such that $\dim\ker(h(s,0))$ and $\dim\ker(h(s,1))$ are constant. Then 
\[\sfl(h(0,\cdot))=\sfl(h(1,\cdot)).\]
\end{cor}
\noindent
Let us note two immediate consequences of the previous corollary that we will need in Section \ref{section-uniqueness} in our discussion of the uniqueness of the spectral flow.

\begin{itemize}
 \item[(H)] Let $h:I\times I\rightarrow\mathcal{CF}^\textup{sa}(H)$ be a homotopy of selfadjoint Fredholm operators such that $h(s,0)$ and $h(s,1)$ are constant for all $s\in I$. Then 
\[\sfl(h(0,\cdot))=\sfl(h(1,\cdot)).\]
 \item[(HI)] Let $h:I\times I\rightarrow\mathcal{CF}^\textup{sa}(H)$ be a homotopy of selfadjoint Fredholm operators such that $h(s,0)$ and $h(s,1)$ are invertible for all $s\in I$. Then 
\[\sfl(h(0,\cdot))=\sfl(h(1,\cdot)).\]
\end{itemize}
A further consequence of Corollary \ref{cor-hominvkernel} is that the spectral flow is invariant under homotopies of based loops, i.e. $\sfl(h(0,\cdot))=\sfl(h(1,\cdot))$ if $h:I\times I\rightarrow\mathcal{CF}^\textup{sa}(H)$ is such that $h(s,0)=h(s,1)=T_0$ for all $s\in I$ and some $T_0\in\mathcal{CF}^\textup{sa}(H)$. We conclude this section by noting that, as a consequence of Theorem \ref{thm-hominv}, the spectral flow actually is invariant under free homotopies of loops, which was shown by a different argument for bounded operators in \cite[Prop. 3.8]{SFLPejsachowicz}.

\begin{cor}\label{cor-freehomotopy}
Let $h:I\times I\rightarrow\mathcal{CF}^\textup{sa}(H)$ be a homotopy of selfadjoint Fredholm operators such that $h(s,0)=h(s,1)$ for all $s\in I$. Then

\[\sfl(h(0,\cdot))=\sfl(h(1,\cdot)).\]
\end{cor}
\noindent
As an application of the previous corollary, we show in the following section that the spectral flow of closed paths in $\mathcal{CF}^\textup{sa}(H)$ is invariant under compact perturbations.


\subsection{A Comparison Theorem under Compact Perturbations}
In this section we study perturbations of gap-continuous paths in $\mathcal{CF}^\textup{sa}(H)$ by paths of relatively compact selfadjoint operators on $H$. First of all, we need to pay attention to the question whether the pointwise sum of two gap-continuous paths of closed operators is gap-continuous. Note that this question does not even make sense in this generality as the sum of two closed operators is not necessarily closed, however, the following theorem is sufficient for our purposes. Let us recall that we denote by $\mathcal{C}(H)$ the set of all closed operators on $H$, which canonically is a metric space with respect to the gap-metric \eqref{gap}. 

\begin{theorem}\label{Kato}
Let $T,S\in\mathcal{C}(H)$ and $A,B\in\mathcal{B}(H)$. Then $T+A, S+B\in\mathcal{C}(H)$ and

\begin{align}\label{gapestimate}
d_G(T+A,S+B)\leq 2\sqrt{2}\sqrt{1+\|A\|^2}\sqrt{1+\|B\|^2}(d_G(T,S)+\|A-B\|).
\end{align}
\end{theorem}  

\begin{proof}
We first note that the set of all closed operators is invariant under additive perturbations of bounded operators (\cite[Prob. III.5.6]{Kato} or \cite[Thm. IV.1.1]{Kato}), which shows the first assertion.\\
Before we begin the proof of \eqref{gapestimate}, we recall an alternative way to compute the gap-metric on $\mathcal{C}(H)$. For $T,S\in\mathcal{C}(H)$, we set

\[\delta(T,S)=\sup_{u\in S_T}{d(u,\gra(S))},\]
where $S_T$ denotes the unit sphere in $\gra(T)$ and $d(u,\gra(S))=\inf_{v\in\gra(S)}\|u-v\|$. By \cite[IV.2]{Kato},

\begin{align}\label{gapequivalent}
d_G(T,S)=\max\{\delta(T,S),\delta(S,T)\}.
\end{align}  
Let now $T,S,A,B$ be as in the assertion of the theorem and consider some $\varphi\in\gra(S+B)$, $\|\varphi\|=1$. Then there exists $u\in\mathcal{D}(S)$ such that $\varphi=(u,(S+B)u)$ and 

\[\|u\|^2+\|(S+B)u\|^2=\|\varphi\|^2=1.\]
We set $r^2:=\|u\|^2+\|Su\|^2>0$ and note for later reference the inequality

\begin{align}\label{ClosedOps-align-pertinequality}
\begin{split}
r^2&=\|u\|^2+\|(S+B)u-Bu\|^2\leq \|u\|^2+2\|(S+B)u\|^2+2\|Bu\|^2\\
&\leq2(\|u\|^2+\|(S+B)u\|^2)+2\|B\|^2\|u\|^2\leq2(1+\|B\|^2),
\end{split}
\end{align}
where we have used that $\|u\|^2+\|(S+B)u\|^2=1$ and so in particular $\|u\|\leq 1$.\\
As $r^{-1}(u,Su)\in S_S$ we see by \eqref{gapequivalent} that for all $\delta'>d_G(S,T)$ 

\begin{align*}
d(r^{-1}(u,Su),\gra(T))\leq \sup_{w\in S_{S}}d(w,\gra(T))\leq d_G(S,T)<\delta'.
\end{align*}
Consequently, $d((u,Su),\gra(T))<r\delta'$ and so there is $v\in\mathcal{D}(T)$ such that 
 
\begin{align}\label{ClosedOps-align-uvinequ}
\|u-v\|^2+\|Su-Tv\|^2<r^2\delta'^2.
\end{align}
Let now $\psi=(v,(T+A)v)\in\gra(T+A)$. Then

\begin{align*}
\|\varphi-\psi\|^2&=\|(u,(S+B)u)-(v,(T+A)v)\|^2=\|u-v\|^2+\|Su-Tv+Bu-Av\|^2\\
&\leq \|u-v\|^2+2\|Su-Tv\|^2+2\|Bu-Av\|^2\\
&\leq2(\|u-v\|^2+\|Su-Tv\|^2)+2\|Bu-Av\|^2.
\end{align*}
By \eqref{ClosedOps-align-uvinequ} and as $\|u\|\leq r$ by the definition of $r$, we get

\begin{align*}
\|\varphi-\psi\|^2&\leq2r^2\delta'^2+2\|Bu-Av\|^2\leq2r^2\delta'^2+2(\|Av-Au\|+\|Au-Bu\|)^2\\
&\leq2r^2\delta'^2+4\|A\|^2\|v-u\|^2+4\|A-B\|^2\|u\|^2\\
&\leq2(1+2\|A\|^2)r^2\delta'^2+4\|A-B\|^2\|u\|^2\\
&\leq2(1+2\|A\|^2)r^2\delta'^2+4r^2\|A-B\|^2,
\end{align*}
which implies by \eqref{ClosedOps-align-pertinequality}

\begin{align*}
\|\varphi-\psi\|^2&\leq4(1+2\|A\|^2)(1+\|B\|^2)\delta'^2+8(1+\|B\|^2)\|A-B\|^2\\
&\leq8(1+\|A\|^2)(1+\|B\|^2)\delta'^2+8(1+\|A\|^2)(1+\|B\|^2)\|A-B\|^2\\
&\leq8(1+\|A\|^2)(1+\|B\|^2)(\delta'^2+\|A-B\|^2)\\
&\leq8(1+\|A\|^2)(1+\|B\|^2)(\delta'+\|A-B\|)^2.
\end{align*}
Consequently,

\begin{align}\label{Families-align-inequalityII}
\|\varphi-\psi\|\leq2\sqrt{2}\sqrt{1+\|A\|^2}\sqrt{1+\|B\|^2}(\delta'+\|A-B\|),
\end{align}
which shows that

\begin{align*}
d(\varphi,\gra(T+A))&=\inf_{\tilde{\psi}\in\gra(T+A)}\|\varphi-\tilde{\psi}\|\leq\|\varphi-\psi\|\\
&\leq 2\sqrt{2}\sqrt{1+\|A\|^2}\sqrt{1+\|B\|^2}(\delta'+\|A-B\|)
\end{align*}
for any $\varphi\in\gra(S+B)$, $\|\varphi\|=1$. Thus

\begin{align*}
\delta(\gra(S+B),\gra(T+A))&=\sup_{\tilde{\varphi}\in S_{S+B}}d(\tilde{\varphi},\gra(T+A))\\
&\leq2\sqrt{2}\sqrt{1+\|A\|^2}\sqrt{1+\|B\|^2}(\delta'+\|A-B\|).
\end{align*}
As $\delta'$ is an arbitrary number greater than $d_G(S,T)$, we obtain

\begin{align*}
\delta(\gra(S+B),\gra(T+A))\leq2\sqrt{2}\sqrt{1+\|A\|^2}\sqrt{1+\|B\|^2}(d_G(S,T)+\|A-B\|).
\end{align*}
Finally, since the right hand side of this inequality is symmetric in $T+A$ and $S+B$, we see by \eqref{gapequivalent} that

\begin{align*}
d_G(T+A,S+B)\leq 2\sqrt{2}\sqrt{1+\|A\|^2}\sqrt{1+\|B\|^2}(d_G(S,T)+\|A-B\|),
\end{align*}
which is the claimed inequality \eqref{gapestimate}.
\end{proof}
\noindent
Let us point out that the argument in the proof of Theorem \ref{Kato} is based on \cite[Thm.IV.2.17]{Kato}, where \eqref{gapestimate} is obtained in the case that $A=B$ with the constant $2$ instead of $2\sqrt{2}$.\\ 
Before coming to our comparison theory for the spectral flow, we briefly want to recall the notion of \textit{relative compactness} for operators. Let $T:\mathcal{D}(T)\subset H\rightarrow H$ and $A:\mathcal{D}(A)\subset H\rightarrow H$ be operators such that $\mathcal{D}(T)\subset\mathcal{D}(A)$. Then $A$ is called $T$-compact if for every bounded sequence $\{u_n\}_{n\in\mathbb{N}}\subset \mathcal{D}(T)$ with $\{T u_n\}_{n\in\mathbb{N}}$ bounded, $\{Au_n\}_{n\in\mathbb{N}}$ contains a convergent subsequence. Let us also recall that for every $T:\mathcal{D}(T)\subset H\rightarrow H$, the domain $\mathcal{D}(T)$ canonically becomes a normed linear space with respect to the \textit{graph norm}

\[\|u\|^2_T=\|u\|^2+\|Tu\|^2,\quad u\in\mathcal{D}(T).\] 
In what follows we write $\mathcal{D}(T)_G$ when we regard $\mathcal{D}(T)$ as a normed linear space with respect to the graph norm. It is readily seen that $\mathcal{D}(T)_G$ is a Hilbert space if and only if $T\in\mathcal{C}(H)$. Finally, the operator $A$ is $T$-compact if and only if $A:\mathcal{D}(T)_G\rightarrow H$ is compact.\\
Now let $\mathcal{A}=\{\mathcal{A}_\lambda\}_{\lambda\in I}$ be a gap-continuous path in $\mathcal{CF}^\textup{sa}(H)$ and let $\mathcal{K}=\{\mathcal{K}_\lambda\}_{\lambda\in I}$ be a norm-continuous path of selfadjoint operators in $\mathcal{B}(H)$ such that each $\mathcal{K}_\lambda$ is $\mathcal{A}_\lambda$-compact. We set $\mathcal{A}+\mathcal{K}=\{\mathcal{A}_\lambda+\mathcal{K}_\lambda\}_{\lambda\in I}$ which is a gap-continuous path in $\mathcal{C}(H)$ by Theorem \ref{Kato}. As $\mathcal{A}_\lambda+\mathcal{K}_\lambda$ is also selfadjoint and Fredholm by well known perturbation theory (see \cite[Thm. V.4.3]{Kato} and \cite[Thm. IV.5.26]{Kato}), it follows that $\mathcal{A}+\mathcal{K}$ actually is a path in $\mathcal{CF}^\textup{sa}(H)$ and so has a spectral flow. Note that it is easy to construct examples of this type with a non-trivial spectral flow. For example, let $\mathcal{A}_\lambda=T\in\mathcal{CF}^\textup{sa}(H)$ be constant and such that $T$ has a compact resolvent. Then the spectrum of $T$ consists of isolated eigenvalues of finite multiplicity (see \cite[Thm. III.6.29]{Kato}). Moreover, it is easy to see that $K_\lambda=\alpha\lambda I_H$ is $T$-compact for any $\alpha>0$. Finally, it follows from \eqref{def-sfl} that the spectral flow of $\mathcal{A}+\mathcal{K}$ is the number of eigenvalues of $T$ in $[-\alpha,0)$. On the other hand, we next see as a consequence of Corollary \ref{cor-freehomotopy} that non-trivial relatively compact perturbations can equally well make no contribution to the spectral flow.

\begin{cor}\label{cor-compper}
If $\mathcal{A}_0=\mathcal{A}_1$ and $\mathcal{K}_0=\mathcal{K}_1$, then

\[\sfl(\mathcal{A}+\mathcal{K})=\sfl(\mathcal{A}).\]
\end{cor}

\begin{proof}
We set $h:I\times I\rightarrow\mathcal{CF}^\textup{sa}(H)$, $h(\lambda,s)=\mathcal{A}_\lambda+s\mathcal{K}_\lambda$ which is a gap-continuous homotopy by Theorem \ref{Kato}. Now the assertion follows from the invariance of the spectral flow under free homotopies of closed paths stated in Corollary \ref{cor-freehomotopy}.
\end{proof}
\noindent 
The previous corollary was recently obtained by the second author in \cite{HomoclinicsInf} as a consequence of a rather technical $K$-theoretic description of the spectral flow for unbounded selfadjoint Fredholm operators. Corollary \ref{cor-compper} shows that this property can also be derived in a more direct way from the definition of the spectral flow.\\
We now come to the main theorem of this section which is a Comparison Principle for the spectral flow of compact perturbations. Let us first recall that there is a partial order on the set of all selfadjoint bounded operators $\mathcal{B}^\textup{sa}(H)$ on $H$, which is defined by

\[A\geq B \quad \text{if}\quad \langle (A-B)u,u\rangle\geq 0,\quad u\in H.\]
The main theorem of this section now reads as follows.

\begin{theorem}\label{thm-comparison}
Let $\mathcal{K}=\{\mathcal{K}_\lambda\}_{\lambda\in I}$ and $\mathcal{K}'=\{\mathcal{K}'_\lambda\}_{\lambda\in I}$ be two paths of operators in $\mathcal{B}^\textup{sa}(H)$ and let $\mathcal{A}=\{\mathcal{A}_\lambda\}_{\lambda\in I}$ be gap-continuous and such that $\mathcal{K}_\lambda$ and $\mathcal{K}'_\lambda$ are $\mathcal{A}_\lambda$-compact for any $\lambda\in I$. If

\[\mathcal{K}'_0\geq \mathcal{K}_0,\quad \mathcal{K}_1\geq \mathcal{K}'_1,\]
then

\[\sfl(\mathcal{A}+\mathcal{K})\geq \sfl(\mathcal{A}+\mathcal{K}').\]
\end{theorem} 

\begin{proof}
We define a homotopy $h:I\times I\rightarrow\mathcal{CF}^\textup{sa}(H)$ by $h(s,\lambda)=\mathcal{A}_\lambda+(1-s)K_\lambda+s K'_\lambda$ and note that $h$ is continuous by Theorem \ref{Kato}. It follows from Theorem \ref{thm-hominv} that

\[\sfl(\mathcal{A}+\mathcal{K})=\sfl(h(\cdot,0))+\sfl(\mathcal{A}+\mathcal{K}')-\sfl(h(\cdot,1)),\]
and so the theorem is proved if we show that 

\begin{align}\label{sflneq0}
\sfl(h(\cdot,0))\geq 0\quad\text{and}\quad \sfl(h(\cdot,1))\leq 0.
\end{align}
To show \eqref{sflneq0}, we need to recall a method for computing the spectral flow that was introduced by Robbin and Salamon in \cite{Robbin} and generalised in \cite{WaterstraatHomoclinics} to the setting that is needed here. Assume that $W\subset H$ is a dense subset which is also a Hilbert space in its own right with a continuous embedding $W\hookrightarrow H$. Let $\mathcal{L}=\{\mathcal{L}_\lambda\}_{\lambda\in I}$ be a path in the normed space of bounded operators $\mathcal{B}(W,H)$, which we assume to be continuously differentiable with respect to the operator norm on the latter space. We also assume that each $\mathcal{L}_\lambda$ is selfadjoint and Fredholm when considered as operator on $H$ with the dense domain $\mathcal{D}(\mathcal{L}_\lambda)=W$. It follows from \cite[Prop. 2.2]{Lesch} that $\mathcal{L}$ is gap-continuous and thus the spectral flow of $\mathcal{L}$ is defined. A parameter value $\lambda^\ast$ is called a \textit{crossing} of $\mathcal{L}$ if $\ker\mathcal{L}_{\lambda^\ast}\neq \{0\}$, and the \textit{crossing form} of a crossing is the quadratic form defined by

\[\Gamma(\mathcal{L},\lambda^\ast)[u]=\langle \dot{\mathcal{L}}_{\lambda^\ast}u,u\rangle,\qquad u\in \ker\mathcal{L}_{\lambda^\ast},\]   
where $\dot{\mathcal{L}}_{\lambda^\ast}$ denotes the derivative of the path $\mathcal{L}$ at $\lambda=\lambda^\ast$. A crossing is called \textit{regular} if $\Gamma(\mathcal{L},\lambda^\ast)$ is non-degenerate. It was shown in \cite{WaterstraatHomoclinics} that, if $\mathcal{L}$ has only regular crossings, then there are only finitely many of them, and the spectral flow of $\mathcal{L}$ is given by

\begin{align}\label{equ-sflcrossings}
-m^-(\Gamma(\mathcal{L},0))+\sum_{\lambda\in(0,1)}{\sgn(\Gamma(\mathcal{L},\lambda))}+m^-(-\Gamma(\mathcal{L},1))
\end{align} 
where $m^-$ denotes the Morse index and $\sgn$ the signature of a quadratic form.\\ 
We now use crossing forms to show \eqref{sflneq0}. Let us first note that $h(\cdot,0)$ and $h(\cdot,1)$ are both differentiable paths of selfadjoint Fredholm operators as above, where the spaces $W$ are $\mathcal{D}(\mathcal{A}_0)$ and $\mathcal{D}(\mathcal{A}_1)$ with the graph norms of these operators. Therefore we can use \eqref{equ-sflcrossings} to compute their spectral flows. Let us consider $h(\cdot,0)$ first. If $s^\ast$ is a crossing of $h(\cdot,0)$, then the crossing form is

\[\Gamma(h(\cdot,0),s^\ast)[u]=\langle(\mathcal{K}'_0-\mathcal{K}_0)u,u\rangle,\quad u\in\ker(h(s^\ast,0)),\]
which is positive definite by assumption. Hence it follows from \eqref{equ-sflcrossings} that $\sfl(h(\cdot,0))\geq 0$. Let us now consider the other case and assume that $s^\ast$ is a crossing of $h(\cdot,1)$. Then the crossing form is

\[\Gamma(h(\cdot,1),s^\ast)[u]=\langle(\mathcal{K}'_1-\mathcal{K}_1)u,u\rangle,\quad u\in\ker(h(s^\ast,1)),\]
which is negative definite by assumption. It follows from \eqref{equ-sflcrossings} that $\sfl(h(\cdot,1))\leq 0$, and so the theorem is shown.
\end{proof}
\noindent



\subsection{Application: Estimating the Spectral Flow for Hamiltonian Systems}
The aim of this section is to apply the comparison theorem to boundary value problems of linear Hamiltonian systems of the form

\begin{equation}\label{Hamiltonian}
\left\{
\begin{aligned}
Ju'(t)+S_\lambda(t)&u(t)=0,\quad t\in I\\
u(0)\in \Lambda_1(\lambda)&,\, u(1)\in \Lambda_2(\lambda),
\end{aligned}
\right.
\end{equation} 
where

\begin{align}\label{J}
J=\begin{pmatrix}
0&-I_n\\
I_n&0
\end{pmatrix}
\end{align}
is the standard symplectic matrix, $\{S_\lambda(t)\}_{(\lambda,t)\in I\times I}$ is a family of symmetric $2n\times 2n$ matrices and $\Lambda_1$, $\Lambda_2$ are paths of Lagrangian subspaces of $\mathbb{R}^{2n}$. Let us recall that a subspace $L\subset\mathbb{R}^{2n}$ is called Lagrangian if $JL=L^\perp$, where the latter denotes the orthogonal complement with respect to the standard scalar product on $\mathbb{R}^{2n}$. The set of all Lagrangian subspaces of $\mathbb{R}^{2n}$ is a $\frac{1}{2}n(n+1)$-dimensional submanifold of the Grassmannian $G_n(\mathbb{R}^{2n})$ of all $n$-dimensional subspaces of $\mathbb{R}^{2n}$. Roughly speaking, the \textit{Maslov index} $\mu_{Mas}(\Lambda_1,\Lambda_2)$ of a pair of paths $(\Lambda_1,\Lambda_2):I\rightarrow\Lambda(n)\times\Lambda(n)$ is an integer-valued relative homotopy invariant that counts the dimensions of non-trivial intersections of $\Lambda_1(\lambda)$ and $\Lambda_2(\lambda)$ whilst the parameter $\lambda$ travels along the unit interval. As this invariant has been studied in numerous references, we will not provide further details and refer to \cite{Cappell} and \cite{MJN}. Let us note, however, that there are different non-equivalent approaches to the Maslov index and here we will always consider $\mu_{Mas}(\Lambda_1,\Lambda_2)$ as constructed in the previous references.\\
In what follows we write $S_\lambda\geq 0$ if $\langle S_\lambda(t) u,u\rangle\geq 0$ and $S_\lambda\leq 0$ if $\langle S_\lambda(t) u,u\rangle\leq 0$ for all $u\in\mathbb{R}^{2n}$ and all $t\in I$. The aim of this section is to prove the following theorem.

\begin{theorem}\label{thm-compHamiltonian}
If either

\begin{itemize}
 \item[(i)] $\mu_{Mas}(\Lambda_1,\Lambda_2)>0$, $S_0\leq 0$ and $S_1\geq 0$, or
 \item[(ii)] $\mu_{Mas}(\Lambda_1,\Lambda_2)<0$, $S_0\geq 0$ and $S_1\leq 0$, 
\end{itemize} 
then there are at least

\[\left\lceil{\frac{|\mu_{Mas}(\Lambda_1,\Lambda_2)|}{n}}\right\rceil\]
parameter values $\lambda$ for which \eqref{Hamiltonian} has a non-trivial solution.
\end{theorem} 
\noindent
We will obtain Theorem \ref{thm-compHamiltonian} from Theorem \ref{thm-comparison} in the remainder of this section. Let us consider on $H=L^2(I,\mathbb{R}^{2n})$ the differential operators 

\[\mathcal{A}_\lambda u=Ju'\]
on the domains

\[\mathcal{D}(\mathcal{A}_\lambda)=\{u\in H^1(I,\mathbb{R}^{2n}):\, u(0)\in\Lambda_1(\lambda),\, u(1)\in\Lambda_2(\lambda)\}.\]
It is well known that $\mathcal{A}_\lambda$ are selfadjoint Fredholm operators (see, e.g., \cite{Cappell}). Moreover, the path $\mathcal{A}=\{\mathcal{A}_\lambda\}_{\lambda\in I}$ is gap-continuous by \cite[Thm. 1.1]{MJN}. As the embedding $H^1(I,\mathbb{R}^{2n})\hookrightarrow L^2(I,\mathbb{R}^{2n})$ is compact by Rellich's Theorem, it follows that the multiplication operators

\[(\mathcal{K}_\lambda u)(t)=S_\lambda(t)u(t)\]
define a path of operators in $\mathcal{B}^\textup{sa}(H)$ such that $\mathcal{K}_\lambda$ is $\mathcal{A}_\lambda$-compact. Hence the spectral flow of $\mathcal{A}+\mathcal{K}$ is defined. Note that a non-vanishing spectral flow of $\mathcal{A}+\mathcal{K}$ implies that there is some $\lambda\in I$ for which \eqref{Hamiltonian} has a non-trivial solution.\\
It was shown by Cappell, Lee and Miller in \cite{Cappell} (see also \cite[Thm. 1.1]{MJN}) that

\[\sfl(\mathcal{A})=\mu_{Mas}(\Lambda_1,\Lambda_2).\]
As either $\mathcal{K}_0\geq 0$, $\mathcal{K}_1\leq 0$ and $\mu_{Mas}(\Lambda_1,\Lambda_2)<0$, or $\mathcal{K}_0\leq 0$, $\mathcal{K}_1\geq 0$ and $\mu_{Mas}(\Lambda_1,\Lambda_2)>0$ by the assumption of Theorem \ref{thm-compHamiltonian}, it follows from Theorem \ref{thm-comparison} that 
\begin{align}\label{equ-sflcomp}
|\sfl(\mathcal{A}+\mathcal{K})|\geq|\sfl(\mathcal{A})|=|\mu_{Mas}(\Lambda_1,\Lambda_2)|.
\end{align}
To finish the proof of Theorem \ref{thm-compHamiltonian}, we need the following lemma about the spectral flow.

\begin{lemma}\label{lemma-sfltechnical}
If $\lambda=\lambda^\ast\in(0,1)$ is the only parameter $\lambda$ for which $\ker(\mathcal{A}_\lambda+\mathcal{K}_\lambda)\neq\{0\}$, then

\[|\sfl(\mathcal{A}+\mathcal{K})|\leq\dim\ker(\mathcal{A}_{\lambda^\ast}+\mathcal{K}_{\lambda^\ast}).\]  
\end{lemma}

\begin{proof}
Let $\varepsilon>0$ and $\delta>0$ be such that 

\[\dim\im\chi_{[-\varepsilon,\varepsilon]}(\mathcal{A}_\lambda+\mathcal{K}_\lambda)=\dim\ker(\mathcal{A}_{\lambda^\ast}+\mathcal{K}_{\lambda^\ast}),\quad |\lambda-\lambda^\ast|\leq\delta.\]
Then by (C), (Z) and the definition of the spectral flow \eqref{def-sfl}

\begin{align*}
|\sfl(\mathcal{A}+\mathcal{K})|&=|\sfl(\mathcal{A}+\mathcal{K}\mid_{[\lambda^\ast-\delta,\lambda^\ast+\delta]})|\\
&=|\dim\im\chi_{[0,\varepsilon]}(\mathcal{A}_{\lambda^\ast+\delta}+\mathcal{K}_{\lambda^\ast+\delta})-\dim\im\chi_{[0,\varepsilon]}(\mathcal{A}_{\lambda^\ast-\delta}+\mathcal{K}_{\lambda^\ast-\delta})|\\
&\leq \max\{\dim\im\chi_{[0,\varepsilon]}(\mathcal{A}_{\lambda^\ast+\delta}+\mathcal{K}_{\lambda^\ast+\delta}),\dim\im\chi_{[0,\varepsilon]}(\mathcal{A}_{\lambda^\ast-\delta}+\mathcal{K}_{\lambda^\ast-\delta})\}\\
&\leq \max\{\dim\im\chi_{[-\varepsilon,\varepsilon]}(\mathcal{A}_{\lambda^\ast+\delta}+\mathcal{K}_{\lambda^\ast+\delta}),\dim\im\chi_{[-\varepsilon,\varepsilon]}(\mathcal{A}_{\lambda^\ast-\delta}+\mathcal{K}_{\lambda^\ast-\delta})\}\\
&= \dim\ker(\mathcal{A}_{\lambda^\ast}+\mathcal{K}_{\lambda^\ast}),
\end{align*}
where we have used that $|a-b|\leq\max\{a,b\}$ if $a,b\geq 0$.
\end{proof}
\noindent
To prove Theorem \ref{thm-compHamiltonian}, we can assume that there are only finitely many $0\leq\lambda_1<\ldots<\lambda_N\leq 1$ such that $\ker(\mathcal{A}_{\lambda_i}+\mathcal{K}_{\lambda_i})\neq\{0\}$. Let $\varepsilon>0$ be sufficiently small such that $\lambda_i$ is the only of these values in the interval $I_i:=[\lambda_i-\varepsilon,\lambda_i+\varepsilon]\cap I$ for $i=1,\ldots,N$. It follows from (C) that 

\[\sfl(\mathcal{A}+\mathcal{K})=\sum^N_{i=1}{\sfl(\mathcal{A}+\mathcal{K}\mid_{I_i})},\]
and consequently by Lemma \ref{lemma-sfltechnical}

\[|\sfl(\mathcal{A}+\mathcal{K})|\leq \sum^N_{i=1}{|\sfl(\mathcal{A}+\mathcal{K}\mid_{I_i})|}\leq\sum^N_{i=1}{\dim\ker(\mathcal{A}_{\lambda_i}+\mathcal{K}_{\lambda_i})}.\]
Now $\ker(\mathcal{A}_{\lambda}+\mathcal{K}_{\lambda})$ is the space of solutions of \eqref{Hamiltonian} and this space has at most dimension $n$ under the given boundary conditions. Hence 

\[|\sfl(\mathcal{A}+\mathcal{K})|\leq \sum^N_{i=1}{\dim\ker(\mathcal{A}_{\lambda_i}+\mathcal{K}_{\lambda_i})}\leq Nn, \]
or, in other words,

\[N\geq \frac{|\sfl(\mathcal{A}+\mathcal{K})|}{n}\geq\frac{|\mu_{Mas}(\Lambda_1,\Lambda_2)|}{n},\]
where we have used \eqref{equ-sflcomp}. As $N$ is the number of parameter values $\lambda$ for which $\ker(\mathcal{A}_{\lambda}+\mathcal{K}_{\lambda})\neq\{0\}$, this shows Theorem \ref{thm-compHamiltonian}.


\section{On the Uniqueness of the Spectral Flow}\label{section-uniqueness}

\subsection{The Uniqueness for Admissible Paths}
In this section we recall Lesch's uniqueness theorem for the spectral flow from \cite{Lesch}. Let $H$ be an infinite dimensional separable complex Hilbert space and denote by $\mathcal{BF}^\textup{sa}(H)$ the set of all bounded selfadjoint Fredholm operators on $H$ with the norm topology. Atiyah and Singer showed in \cite{AtiyahSinger} that $\mathcal{BF}^\textup{sa}(H)$ consists of three path-components

\[\mathcal{BF}^\textup{sa}(H)=\mathcal{BF}^\textup{sa}_+(H)\cup \mathcal{BF}^\textup{sa}_-(H)\cup \mathcal{BF}^\textup{sa}_\ast(H).\]  
An operator $T\in \mathcal{BF}^\textup{sa}(H)$ is in $\mathcal{BF}^\textup{sa}_\pm(H)$ if its essential spectrum is contained in the positive or the negative half-line, respectively. The complement of these sets, above denoted by $\mathcal{BF}^\textup{sa}_\ast(H)$, are those operators which have positive and negative essential spectrum. The two components $\mathcal{BF}^\textup{sa}_\pm(H)$ are topologically trivial, whereas $\mathcal{BF}^\textup{sa}_\ast(H)$ is a classifying space for the odd $K$-theory functor. Hence any homotopy invariant for paths on $\mathcal{BF}^\textup{sa}_\pm(H)$ can only depend on the endpoints of the paths. On the other hand, $\mathcal{BF}^\textup{sa}_\ast(H)$ has an infinitely cyclic fundamental group and Atiyah, Patodi and Singer constructed in \cite{APS} the spectral flow as an isomorphism between $\pi_1(\mathcal{BF}^\textup{sa}_\ast(H))$ and the integers. Later Phillips gave in \cite{Phillips} an analytic definition for general non-closed paths in $\mathcal{BF}^\textup{sa}_\ast(H)$ by using the formula \eqref{def-sfl} that was later generalised to gap-continuous paths in \cite{UnbSpecFlow}.\\
There is a straightforward way to extend the definition of the spectral flow from bounded to unbounded operators by using the \textit{Riesz transform} 

\begin{align}\label{Riesz}
F:\mathcal{CF}^\textup{sa}(H)\rightarrow\mathcal{BF}^\textup{sa}(H),\qquad F(T)=T(I_H+T^2)^{-\frac{1}{2}}.
\end{align}    
If we require that $F$ is an isometric embedding, then we obtain a metric $d_R$ on $\mathcal{CF}^\textup{sa}(H)$ which is called the \textit{Riesz metric}. Now every path that is continuous with respect to this metric has a spectral flow defined by the spectral flow of the Riesz transform of the path. It was shown by Nicolaescu in \cite{Nicolaescu} that the topology induced by the gap-metric $d_G$ is weaker than the one induced by the Riesz-metric $d_R$. Lesch proved in \cite[Thm. 5.10]{Lesch} that the natural inclusion $\mathcal{BF}^\textup{sa}(H)\hookrightarrow (\mathcal{CF}^\textup{sa}(H),d_R)$ is a homotopy equivalence. Hence $(\mathcal{CF}^\textup{sa}(H),d_R)$ consists of three path components

\begin{align}\label{componentsRiesz}
(\mathcal{CF}^\textup{sa}(H),d_R)=\mathcal{CF}^\textup{sa}_+(H)\cup \mathcal{CF}^\textup{sa}(H)_-\cup \mathcal{CF}^\textup{sa}(H)_\ast
\end{align} 
of which $\mathcal{CF}^\textup{sa}(H)_\pm$ are topologically trivial whereas $\mathcal{CF}^\textup{sa}_\ast(H)$ has an infinitely cyclic fundamental group.\\
In the previous sections we only considered $\mathcal{CF}^\textup{sa}(H)$ with the gap metric $d_G$. As this metric is weaker than $d_R$, we see that it is a weaker assumption to require that a path is continuous with respect to $d_G$. Lesch proved in \cite[Prop. 5.8]{Lesch} the remarkable result that $(\mathcal{CF}^\textup{sa}(H),d_G)$ is path-connected, and so there is no division into paths components as in \eqref{componentsRiesz}. Joachim showed in \cite{Joachim} that the whole space $(\mathcal{CF}^\textup{sa}(H),d_G)$ is a classifying space for the odd K-theory functor. In particular, the fundamental group $\pi_1(\mathcal{CF}^\textup{sa}(H))$ is infinitely cyclic.\\
Let now $G\mathcal{C}^\textup{sa}(H)$ denote the set of all invertible elements of $\mathcal{C}^\textup{sa}(H)$, and $G\mathcal{B}^\textbf{sa}(H)$ be the set of all invertible elements of $\mathcal{B}^\textup{sa}(H)$. In what follows we want to treat all the above topologically non-trivial operator sets simultanously. Therefore, we let the topological space pair $(X,Y)$ be either $(\mathcal{CF}^\textup{sa}(H),G\mathcal{C}^\textup{sa}(H))$ with the gap topology, or $(\mathcal{CF}^\textup{sa}_\ast(H),G\mathcal{C}^\textup{sa}(H)\cap \mathcal{CF}^\textup{sa}_\ast(H))$ with the Riesz topology, or $(\mathcal{BF}^\textup{sa}_\ast(H),G\mathcal{B}^\textup{sa}(H)\cap \mathcal{BF}^\textup{sa}_\ast(H))$ with the norm topology. We first note for later reference the following important theorem.

\begin{theorem}\label{thm-gap-pathconnected}
The set $Y$ is path-connected.
\end{theorem}

\begin{proof}
This was proved in \cite[Prop. 5.8]{Lesch} for $G\mathcal{C}^\textup{sa}(H)$ with the gap topology, and it is an easy exercise for $G\mathcal{B}^\textup{sa}(H)\cap \mathcal{BF}^\textup{sa}_\ast(H)$ with the norm topology. The remaining case then follows from the already mentioned fact that the inclusion $\mathcal{BF}^\textup{sa}(H)\hookrightarrow(\mathcal{CF}^\textup{sa}(H),d_R)$ is a homotopy equivalence \cite[Thm. 5.10]{Lesch}.
\end{proof}
\noindent
In what follows $\Omega(X)$ stands for the set of all paths in $X$ and $\Omega(X,Y)$ denotes the set of those elements in $\Omega(X)$ which have endpoints in $Y$. We will below consider maps $\mu:\Omega(X)\rightarrow\mathbb{Z}$ or $\mu:\Omega(X,Y)\rightarrow\mathbb{Z}$ with the following properties:

\begin{itemize}
 \item[(Z)] If $\mathcal{A}=\{\mathcal{A}_\lambda\}_{\lambda\in I}\in\Omega(X)$ is such that $\mathcal{A}_\lambda\in Y$ for all $\lambda\in I$, then $\mu(\mathcal{A})=0$. 
 \item[(C)] If $\mathcal{A}^1$ and $\mathcal{A}^2$ are in $\Omega(X)$ such that $\mathcal{A}^1_1=\mathcal{A}^2_0$, then 
 \[\mu(\mathcal{A}^1\ast\mathcal{A}^2)=\mu(\mathcal{A}^1)+\mu(\mathcal{A}^2).\]
 \item[(H)] Let $h:I\times I\rightarrow X$ be a homotopy of selfadjoint Fredholm operators such that $h(\lambda,0)$ and $h(\lambda,1)$ are constant for all $\lambda\in I$. Then 
\[\mu(h(0,\cdot))=\mu(h(1,\cdot)).\]
 \item[(HI)] Let $h:I\times I\rightarrow X$ be a homotopy of selfadjoint Fredholm operators such that $h(\lambda,0), h(\lambda,1)\in Y$ for all $\lambda\in I$. Then 
\[\mu(h(0,\cdot))=\mu(h(1,\cdot)).\]
\end{itemize}
Note that we have seen in Section \ref{section-sfldef} and Section \ref{section-hominv} that the spectral flow has all these properties. Before we state Lesch's Theorem, let us introduce a further property.

\begin{itemize}
 \item[(NP)] There is a bounded selfadjoint operator $T_0$ such that $\sigma(T_0)=\sigma_{ess}(T_0)=\{-1,1\}$, and a rank one orthogonal projection $P$ such that $(I-P)T_0(I-P):\ker(P)\rightarrow\ker(P)$ is invertible and
 
 \[\mu(\mathcal{A}^{NP})=1,\]
where $\mathcal{A}^{NP}_\lambda=(\lambda-\frac{1}{2})P+(I-P)T_0(I-P):H\rightarrow H$.   
\end{itemize}
Note that (NP) is satisfied by the spectral fow. Indeed, we let $P_+, P_-$ and $P_0$ be three complementary orthogonal projections such that $P_+$ and $P_-$ have infinite dimensional kernel and image, $\dim(\im P_0)=1$ as well as $P_++P_-+P_0=I_H$. Then we have for $T_0=P_0+P_+-P_-$ and $P=P_0$

\[(\lambda-\frac{1}{2})P+(I-P)T_0(I-P)=(\lambda-\frac{1}{2})P_0+P_+-P_-,\] 
and it is readily seen from \eqref{def-sfl} that this path has a spectral flow of 1.\\
Now Lesch's Uniqueness Theorem reads as follows.

\begin{theorem}\label{Lesch}
Every map 

\[\mu:\Omega(X,Y)\rightarrow\mathbb{Z}\]
that satisfies (C), (HI) and (NP) is the spectral flow \eqref{def-sfl}.
\end{theorem}
\noindent
Let us finally mention that Ciriza, Fitzpatrick and Pejsachowicz obtained in \cite{JacoboUniqueness} a uniqueness theorem for the spectral flow on all of $\mathcal{BF}^\textup{sa}(H)$, where as in Theorem \ref{Lesch} only paths with invertible endpoints were considered. They used instead of the concatenation a direct sum axiom and showed that every path in $\mathcal{BF}^\textup{sa}(H)$ is homotopic to the direct sum of a path on a finite dimensional subspace of $H$ and a path having vanishing spectral flow. As normalisation they required that the spectral flow on a finite dimensional Hilbert space is given by the difference of the number of negative eigenvalues at the endpoints of a path. The reader can easily check that this is in accordance with \eqref{def-sfl}.

\subsection{The Uniqueness Theorem}
The aim of this section is to show the uniqueness of the spectral flow on the bigger set $\Omega(X)$ of all paths in $X$. 
In our uniqueness theorem, we will need the following \textit{normalisation property}. Let us recall that $0$ is either in the resolvent set of a selfadjoint Fredholm operator or it is an isolated eigenvalue of finite multiplicity. 

\begin{enumerate}
\item[(N)] Let $T\in X$ and set $\delta(T)=\frac{1}{2}\min\{|\lambda|:\, 0\neq \lambda\in\sigma(T)\}$. Consider the path 

\[\mathcal{A}^0_t=T+tI_H,\quad t\in[-\delta(T),\delta(T)].\]
Then

\[\mu(\mathcal{A}^0)=\dim\ker(T),\qquad \mu(\mathcal{A}^0\mid_{[0,\delta(T)]})=0.\]
\end{enumerate}
\noindent
Note that (N) holds for the spectral flow, which follows immediately from the definition \eqref{def-sfl}. The following theorem is our main result of this section.
 
\begin{theorem}\label{thm-uniqueness}
Let $\mu:\Omega(X)\rightarrow\mathbb{Z}$ be a map such that (C), (H), (Z) and (N) hold. Then

\[\mu(\mathcal{A})=\sfl(\mathcal{A}),\quad \mathcal{A}\in \Omega(X).\] 
\end{theorem}

\begin{proof}
Let $T\in X$ be an operator having a one-dimensional kernel and let $\mathcal{A}^0$ be the corresponding path in (N). To simplify notation, we set $T_0:=\mathcal{A}^0_{-\delta(T)}=T-\delta(T)I_H$.\\
By (C) and (H) it is clear that $\mu$ and $\sfl$ induce homomorphisms

\begin{align}\label{fundiso}
\mu,\sfl:\pi_1(X,T_0)\rightarrow\mathbb{Z}.
\end{align}
Let us recall from the previous section that the fundamental group $\pi_1(X,T_0)$ is infinitely cyclic. By Theorem \ref{thm-gap-pathconnected}, there is a path $\mathcal{A}^1$ in $Y$ connecting $\mathcal{A}^0_{\delta(T)}$ to $T_0$. Now the concatenation $\mathcal{A}^0\ast\mathcal{A}^1$ is an element in $\pi_1(X,T_0)$ and we obtain from (C) and (Z)

\begin{align*}
\mu(\mathcal{A}^0\ast\mathcal{A}^1)&=\mu(\mathcal{A}^0)+\mu(\mathcal{A}^1)=\mu(\mathcal{A}^0)=\dim\ker(T)\\
&=\sfl(\mathcal{A}^0)+\sfl(\mathcal{A}^1)=\sfl(\mathcal{A}^0\ast\mathcal{A}^1).
\end{align*} 
As $\dim\ker(T)=1$, this firstly shows that $\mathcal{A}^0\ast\mathcal{A}^1$ is a generator of the infinitely cyclic group $\pi_1(X,T_0)$, and secondly that $\mu$ and $\sfl$ have the same value on it. Hence the maps in \eqref{fundiso} coincide.\\
Let now $\mathcal{A}\in\Omega(X)$ be an arbitrary path in $X$. Let us first consider the endpoints $\mathcal{A}_0$ and $\mathcal{A}_1$. We set 

\[\mathcal{B}_t=\mathcal{A}_0+tI_H,\, t\in[-\delta(\mathcal{A}_0),0],\qquad \mathcal{C}_t=\mathcal{A}_1+tI_H,\, t\in[0,\delta(\mathcal{A}_1)].\]
It follows from (N) that $\sfl(\mathcal{C})=\mu(\mathcal{C})=0$, and by using (N) and (C), that $\sfl(\mathcal{B})=\mu(\mathcal{B})=\dim\ker(\mathcal{A}_0)$.\\
Let now $\mathcal{A}^1$ be a path in $Y$ connecting $T_0$ to the initial point of $\mathcal{B}$ and let $\mathcal{A}^2$ be another path of this type connecting the endpoint of $\mathcal{C}$ to $T_0$. It follows from the first part of our proof, (C) and (Z) that

\begin{align*}
\mu(\mathcal{A})&=\mu(\mathcal{B}\ast\mathcal{A}\ast\mathcal{C})-\dim\ker(\mathcal{A}_0)=\mu(\mathcal{A}^1\ast\mathcal{B}\ast\mathcal{A}\ast\mathcal{C}\ast\mathcal{A}^2)-\dim\ker(\mathcal{A}_0)\\
&=\sfl(\mathcal{A}^1\ast\mathcal{B}\ast\mathcal{A}\ast\mathcal{C}\ast\mathcal{A}^2)-\dim\ker(\mathcal{A}_0)=\sfl(\mathcal{B}\ast\mathcal{A}\ast\mathcal{C})-\dim\ker(\mathcal{A}_0)\\
&=\sfl(\mathcal{A}),
\end{align*}
and so the claim is shown.
\end{proof}
\noindent
The reader may have noticed that, apart from a different normalisation property, we have replaced (HI) in Lesch's Theorem by (H) and (Z). The following lemma shows that the normalisation property actually is the only difference in the assumptions of the theorems. 

\begin{lemma}
If $\mu:\Omega(X)\rightarrow\mathbb{Z}$ is a map such that (C) and (N) hold, then (HI) is satisfied if and only if (H) and (Z) hold. 
\end{lemma}

\begin{proof}
We below use without further reference that if $\mathcal{A}\in\Omega(X)$ is a constant path, then by (C)

\[\mu(\mathcal{A})=\mu(\mathcal{A}\ast\mathcal{A})=\mu(\mathcal{A})+\mu(\mathcal{A}),\]
showing that $\mu(\mathcal{A})=0$.\\
Let us first assume that (HI) is satisfied and let $\mathcal{A}$ be a path of invertible operators. Then $\mu(\mathcal{A})=\mu(\widetilde{A})$, where $\widetilde{A}_t=\mathcal{A}_0$ for all $t\in I$. Hence $\mu(\mathcal{A})=0$ and (Z) is satisfied. Let now $h:I\times I\rightarrow X$ be a homotopy such that $h(s,0)$ and $h(s,1)$ are constant. By (N), we can concatenate $h(s,0)$ with a path $\mathcal{A}$ and $h(s,1)$ with a path $\mathcal{B}$ such that $\mathcal{A}\ast h(s,\cdot)\ast\mathcal{B}$ has invertible endpoints for all $s$ and 

\[\mu(\mathcal{A})=\dim\ker h(s,0),\qquad \mu(\mathcal{B})=0.\]
Hence by (HI) and (C)

\[\mu(h(0,\cdot))=\mu(\mathcal{A}\ast h(0,\cdot)\ast\mathcal{B})-\dim\ker(h(s,0))=\mu(\mathcal{A}\ast h(1,\cdot)\ast\mathcal{B})-\dim\ker(h(s,0))=\mu(h(1,\cdot)),\]
which shows (H).\\
For the converse, let us first note that $\mu(\mathcal{A})=-\mu(\mathcal{A}')$, where $\mathcal{A}'$ denotes the reverse path. Indeed, this follows as $\mathcal{A}\ast\mathcal{A}'$ is homotopic to a constant path by a homotopy with fixed endpoints. Then (C) implies that $0=\mu(\mathcal{A}\ast\mathcal{A}')=\mu(\mathcal{A})+\mu(\mathcal{A}')$. Let now $h:I\times I\rightarrow X$ be a homotopy such that $h(s,0), h(s,1)\in Y$ for all $s\in I$. As $h(0,\cdot)\ast h(\cdot,1)\ast h(1,\cdot)'\ast h(\cdot,0)'$ is homotopic to a constant path by a homotopy with fixed endpoints, we obtain from (H), (C) and (Z)

\begin{align*}
0&=\mu(h(0,\cdot)\ast h(\cdot,1)\ast h(1,\cdot)'\ast h(\cdot,0)')=\mu(h(0,\cdot))+\mu(h(\cdot,1))+\mu(h(1,\cdot)')+\mu(h(\cdot,0)')\\
&=\mu(h(0,\cdot))+\mu(h(1,\cdot)')=\mu(h(0,\cdot))-\mu(h(1,\cdot))
\end{align*}   
which shows (HI).
\end{proof}
\noindent


\subsection{Application: Spectral Flow and Maslov Index}


\subsubsection{The Maslov Index in Symplectic Hilbert Spaces}
We begin this section by giving a brief introduction into symplectic Hilbert spaces, where we refer for further details to Furutani's review \cite{Furutani}. \\
In this section we let $E$ be a real separable Hilbert space with scalar product $\langle\cdot,\cdot\rangle_E$, and we assume that there is a bounded linear operator $J:E\rightarrow E$ such that $J^2=-I_E$ and $J^\ast=-J$, where $J^\ast$ denotes the adjoint of $J$. We call the pair $(E,J)$ a symplectic Hilbert space and we set $\omega_0(u,v)=\langle Ju,v\rangle_E$, $u,v\in E$, which is a non-degenerate skew-symmetric bounded bilinear form. A subspace $L\subset E$ is called \textit{Lagrangian} if $L^\perp=JL$, where $L^\perp$ denotes the orthogonal complement of $L$ with respect to the scalar product on $E$. The set $\Lambda(E)$ of all Lagrangian subspaces in $E$ is a smooth Banach manifold. Note that every Lagrangian subspace is closed and so there is a unique orthogonal projection $P_L:E\rightarrow E$ onto $L$. The topology of $E$ is also induced by the metric

\[d(L_1,L_2)=\|P_{L_1}-P_{L_2}\|_{\mathcal{L}(H\times H)},\quad L_1,L_2\in\Lambda(E).\]
Let us recall that the idea of the Maslov index for paths of Lagrangian subspaces of $\mathbb{R}^{2n}$ is to count the dimensions of their intersections. Note that this obviously cannot be done in the above setting as here the intersection of two Lagrangian subspaces can be of infinite dimension. Two closed subspaces $L_1,L_2\subset E$ are called a \textit{Fredholm pair} if

\[\dim(L_1\cap L_2)<\infty\,\, \text{ and }\, \codim(L_1+L_2)<\infty.\]
It is often required in addition that $L_1+L_2$ is closed which, however, is redundant as explained, e.g., in \cite{BoossFurutani}. We now fix a Lagrangian subspace $\Lambda_0\in\Lambda(E)$ and set   

\[\mathcal{FL}_{L_0}(E)=\{L\in\Lambda(E):\,(L,L_0)\,\,\text{Fredholm}\}.\]
It was shown by Booss-Bavnbek and Furutani in \cite{BoossFurutaniI} that there is a Maslov index $\mu_{Mas}(\Lambda,L_0)$ for paths $\Lambda=\{\Lambda(\lambda)\}_{\lambda\in I}$ in $\mathcal{FL}_{L_0}(E)$ which has the interpretation that it counts the dimensions of intersections of $\Lambda(\lambda)$ and $L_0$ whilst $\lambda$ travels alomg the interval $I$. Moreover, the Maslov index has the following properties (see \cite{Furutani}):

\begin{enumerate}
\item[(i)] If $\Lambda(\lambda)\cap L_0=\{0\}$ for all $\lambda\in I$, then $\mu_{Mas}(\Lambda,L_0)=0$.
\item[(ii)] The Maslov index is additive under the concatenation of paths, i.e.
\[\mu_{Mas}(\Lambda_1\ast\Lambda_2,L_0)=\mu_{Mas}(\Lambda_1,L_0)+\mu_{Mas}(\Lambda_2,L_0)\]
if $\Lambda_1,\Lambda_2:I\rightarrow\mathcal{FL}_{L_0}(E)$ are two paths such that $\Lambda_1(1)=\Lambda_2(0)$.
\item[(iii)] If $\Lambda:I\times I\rightarrow\mathcal{FL}_{L_0}(E)$ is a homotopy such that $\Lambda(s,0)$ and $\Lambda(s,1)$ are constant for all $s\in I$, then
\[\mu_{Mas}(\Lambda(0,\cdot),L_0)=\mu_{Mas}(\Lambda(1,\cdot),L_0).\]
\end{enumerate}
Finally, let us recall from \cite[\S 3.4]{Furutani} that the Maslov index for differentiable paths $\Lambda=\{\Lambda(\lambda)\}_{\lambda\in I}$ in $\mathcal{FL}_{L_0}(E)$ can easily be computed in cases when there is only one single parameter value $\lambda_0$ for which $\Lambda(\lambda_0)\cap L_0\neq\{0\}$. Let $L\in\Lambda(E)$ be such that $L\cap\Lambda(\lambda_0)=\{0\}$. Then $\Lambda(\lambda)\cap L=\{0\}$ for all $\lambda$ which are sufficiently close to $\lambda_0$ and there is a differentiable path of bounded linear operators $\phi_\lambda:\Lambda(\lambda_0)\rightarrow L$ such that 

\[\Lambda(\lambda)=\{u+\phi_\lambda(u):\,u\in\Lambda(\lambda_0)\}.\]
The Maslov index of $\Lambda$ is given by the signature of

\begin{align}\label{crossing}
\mathcal{Q}:\Lambda(\lambda_0)\cap L_0\rightarrow\mathbb{R},\quad \mathcal{Q}[u]=\frac{d}{d\lambda}\mid_{\lambda=\lambda_0}\omega(u,\phi_\lambda(u))
\end{align}
if this quadratic form is non-degenerate. Moreover, if \eqref{crossing} is positive definite, then

\begin{align}\label{Maslovsplit}
\mu_{Mas}(\Lambda\mid_{[0,\lambda_0]},L_0)=\sgn \mathcal{Q}=\dim(\Lambda(\lambda_0)\cap L_0),\qquad \mu_{Mas}(\Lambda\mid_{[\lambda_0,1]},L_0)=0
\end{align}

\subsubsection{Spectral Flow via the Maslov Index}
Let now $H$ be a complex separable Hilbert space with scalar product $\langle\cdot,\cdot\rangle$. The realification $H_\mathbb{R}$ of $H$ is a real Hilbert space with scalar product $\langle\cdot,\cdot\rangle_\mathbb{R}=\operatorname{Re}\langle\cdot,\cdot\rangle$.  We set $E=H_\mathbb{R}\times H_{\mathbb{R}}$ which is a symplectic Hilbert space with respect to 

\[J=\begin{pmatrix}
0&-I_{H_{\mathbb{R}}}\\
I_{H_{\mathbb{R}}}&0
\end{pmatrix}\]
Note that the norms of $H$ and $H_\mathbb{R}$ coincide, and so the topologies of $E$ and $H\times H$ are the same.

\begin{lemma}
If $T\in\mathcal{C}^\textup{sa}(H)$, then $\gra(T)\in\Lambda(E)$. 
\end{lemma}

\begin{proof}
We have to show that $J\gra(T)=\gra(T)^\perp$, where the orthogonal complement is with respect to the scalar product of $E$. As

\[J\gra(T)=\{(-Tu,u)\in H\times H:\, u\in\mathcal{D}(T)\},\]
and 

\[\langle (-Tu,u),(v,Tv)\rangle_{H\times H}=-\langle Tu,v\rangle+\langle u,Tv\rangle=0,\quad u,v\in\mathcal{D}(T),\]
it follows that 

\[\langle (-Tu,u),(v,Tv)\rangle_E=-\operatorname{Re}(\langle Tu,v\rangle)+\operatorname{Re}(\langle u,Tv\rangle)=0,\quad u,v\in\mathcal{D}(T),\]
and so
$J\gra(T)\subset \gra(T)^\perp$. Conversely, if $(v,w)\in\gra(T)^\perp$, then

\begin{align*}
\operatorname{Re}(\langle v,u\rangle)+\operatorname{Re}(\langle w,Tu\rangle)=0,\quad u\in\mathcal{D}(T),
\end{align*}   
which also shows that

\begin{align*}
\operatorname{Im}(\langle v,u\rangle)+\operatorname{Im}(\langle w,Tu\rangle)=\operatorname{Re}(\langle v,iu\rangle)+\operatorname{Re}(\langle w,iTu\rangle)=0,\quad u\in\mathcal{D}(T),
\end{align*}
and so

\begin{align}\label{scalarprodJperp}
\langle v,u\rangle+\langle w,Tu\rangle=0,\quad u\in\mathcal{D}(T).
\end{align} 
Thus, $u\mapsto \langle w,Tu\rangle$ is bounded on $\mathcal{D}(T)$ and so $w\in\mathcal{D}(T^\ast)=\mathcal{D}(T)$. Hence, by \eqref{scalarprodJperp},

\[\langle v+Tw,u\rangle=0,\quad u\in\mathcal{D}(T),\]
which shows that $v=-Tw$. Consequently, $(v,w)=(-Tw,w)=J(w,Tw)\in J\gra(T)$. 
\end{proof}
\noindent
In what follows, we set $\Lambda_0=H_\mathbb{R}\times\{0\}$ which is an element of $\Lambda(E)$.

\begin{lemma}\label{CFimpFL}
If $T\in\mathcal{CF}^\textup{sa}(H)$, then $\gra(T)\in\mathcal{FL}_{\Lambda_0}(E)$.
\end{lemma}

\begin{proof}
As $\gra(T)\cap\Lambda_0=\ker(T)\times\{0\}$, we see that this intersection is of finite dimension as $T$ is Fredholm. Moreover, if $V\subset H$ is such that $\im(T)\oplus V=H$, then for any $v\in H$, there are $w\in\mathcal{D}(T)$ and $v_1\in V$ such that $v=Tw+v_1$. Consequently, if $(u,v)\in H\times H$, then

\[(u,v)=(u-w,0)+(w,Tw)+(0,v_1)\in ((H\times\{0\})+ \gra(T))\oplus (\{0\}\times V)\]
showing that $\dim V$ is the codimension of $\Lambda_0+ \gra(T)$, which is finite as $T$ is Fredholm. Hence $(\gra(T),\Lambda_0)$ is a Fredholm pair and so the assertion is shown by the previous lemma.   
\end{proof}
Note that if $\{\mathcal{A}_\lambda\}_{\lambda\in I}$ is a gap-continuous path in $\mathcal{CF}^\textup{sa}(H)$, then $\{\gra(\mathcal{A}_\lambda)\}_{\lambda\in I}$ is continuous in $\Lambda(E)$ by the definition of the metric on the latter space. Consequently, the Maslov index of $\{\gra(\mathcal{A}_\lambda)\}_{\lambda\in I}$ with respect to $\Lambda_0$ is well-defined by Lemma \ref{CFimpFL}. The following theorem is the main result of this section.

\begin{theorem}
For any path $\mathcal{A}=\{\mathcal{A}_\lambda\}_{\lambda\in I}$ in $\mathcal{CF}^\textup{sa}(H)$,

\[\sfl(\mathcal{A})=\mu_{Mas}(\{\gra(\mathcal{A}_\lambda)\}_{\lambda\in I},\Lambda_0).\] 
\end{theorem}

\begin{proof}
We define a map 

\[\mu:\Omega(\mathcal{CF}^\textup{sa}(H))\rightarrow\mathbb{Z},\quad \{\mathcal{A}_\lambda\}_{\lambda\in I}\mapsto \mu_{Mas}(\{\gra(\mathcal{A}_\lambda)\}_{\lambda\in I},\Lambda_0)\]
and our aim is to show the properties (C), (H), (Z) and (N) in Theorem \ref{thm-uniqueness}. We note at first that (C) is an immediate consequence of (ii) from above. Moreover, (H) follows from (iii) and the fact that any gap-continuous homotopy of two paths in $\mathcal{CF}^\textup{sa}(H)$ induces a continuous homotopy in $\Lambda(E)$ by the definition of the metric on the latter space. To see (Z), we just need to note that the intersection $\gra(\mathcal{A}_\lambda)\cap(H\times\{0\})=\ker(\mathcal{A}_\lambda)\times\{0\}$ is isomorphic to the kernel of $\mathcal{A}_\lambda$. Hence  $\gra(\mathcal{A}_\lambda)\cap(H\times\{0\})=\{0\}$ if $\ker(\mathcal{A}_\lambda)=\{0\}$ and so (Z) holds by (i).\\
Finally, let us show (N). We set for $T\in\mathcal{CF}^\textup{sa}(H)$ as in (N) $\mathcal{A}^0_\lambda=T+\lambda I_H$ which is a path in $\mathcal{CF}^\textup{sa}(H)$. Moreover, we set $\Lambda(\lambda)=\gra(\mathcal{A}^0_\lambda)$ and note that $\{0\}\times H$ is transversal to $\Lambda(\lambda)$ for all $\lambda$. Set

\[\phi_\lambda:\gra(\mathcal{A}^0_0)=\gra(T)\rightarrow\{0\}\times H,\quad (u,Tu)\mapsto (0,\lambda u),\]
which is a path of bounded operators. Note that

\[\gra(\mathcal{A}^0_\lambda)=\{(u,Tu)+\phi_\lambda(u,Tu):\, u\in\mathcal{D}(T)\},\]
and so $\phi_\lambda$ can be used to compute the Maslov index of $(\gra(\mathcal{A}^0_\cdot),\Lambda_0)$ at $\lambda=0$. The crossing form is

\begin{align*}
\mathcal{Q}((u,Tu),(v,Tv))&=\frac{d}{d\lambda}\mid_{\lambda=0}\omega_0((u,Tu),\phi_\lambda(v,Tv))=\frac{d}{d\lambda}\mid_{\lambda=0}\langle(-Tu,u),(0,\lambda v)\rangle\\
&=\langle u,v\rangle,\quad u,v\in\mathcal{D}(T).
\end{align*} 
Hence the signature of the restriction of $\mathcal{Q}$ to 

\[\gra(T)\cap(H\times\{0\})=\{(u,0)\in\gra(T):\, u\in\ker(T)\}\]
is the dimension of $\ker(T)$, and we see from \eqref{Maslovsplit}

\begin{align*}
\mu(\mathcal{A}^0)=\dim\ker(T),\quad\mu(\mathcal{A}\mid_{[0,\delta(T)]})=0,
\end{align*} 
where we use that $\gra(\mathcal{A}_\lambda)\cap(H\times\{0\})=\{0\}$ for $0<|\lambda|<\delta(T)$. Hence all assumptions of Theorem \ref{thm-uniqueness} are shown and so $\mu(\mathcal{A})=\sfl(\mathcal{A})$ for all paths $\mathcal{A}$ in $\mathcal{CF}^\textup{sa}(H)$.
\end{proof}


\vspace*{1.3cm}

\begin{minipage}{1.2\textwidth}
\begin{minipage}{0.4\textwidth}
Maciej Starostka\\
Ruhr-Universit{\"a}t Bochum\\
and\\
Gdansk University of Technology\\
maciejstarostka@gmail.com
\end{minipage}
\hfill
\begin{minipage}{0.6\textwidth}
Nils Waterstraat\\
Martin-Luther-Universit\"at Halle-Wittenberg\\
Naturwissenschaftliche Fakult\"at II\\
Institut f\"ur Mathematik\\
06099 Halle (Saale)\\
Germany\\
nils.waterstraat@mathematik.uni-halle.de

\end{minipage}
\end{minipage}


\thebibliography{99}


 
\bibitem{AtiyahSinger} M.F. Atiyah, I.M. Singer, \textbf{Index Theory for skew--adjoint Fredholm operators}, Inst. Hautes Etudes Sci. Publ. Math. \textbf{37}, 1969, 5--26 
 
\bibitem{APS} M.F. Atiyah, V.K.  Patodi, I.M.  Singer, \textbf{Spectral asymmetry and Riemannian geometry III},
 Math. Proc. Cambridge Philos. Soc.  \textbf{79}, 1976, 71--99

\bibitem{BoossFurutaniI} B. Booss-Bavnbek, K. Furutani, \textbf{The Maslov index: a functional analytical definition and the spectral flow formula}, Tokyo J. Math. \textbf{21}, 1998, 1--34.
 
\bibitem{BoossFurutani} B. Booss-Bavnbek, K. Furutani, \textbf{Symplectic functional analysis and spectral invariants}, Geometric aspects of partial differential equations (Roskilde, 1998), 53--83, Contemp. Math., 242, Amer. Math. Soc., Providence, RI,  1999

\bibitem{UnbSpecFlow} B. Booss-Bavnbek, M. Lesch, J. Phillips, \textbf{Unbounded Fredholm operators and spectral flow}, Canad. J. Math. \textbf{57}, 2005, 225--250


\bibitem{Cappell} S.E. Cappell, R. Lee, E.Y. Miller, \textbf{On the Maslov index}, Comm. Pure Appl. Math. \textbf{47},  1994, 121--186

\bibitem{JacoboUniqueness} E. Ciriza, P.M. Fitzpatrick, J. Pejsachowicz, \textbf{Uniqueness of Spectral Flow}, Math. Comp. Mod. \textbf{32}, 2000, 1495--1501


\bibitem{SFLPejsachowicz}  P.M. Fitzpatrick, J. Pejsachowicz, L. Recht, \textbf{Spectral Flow and Bifurcation of Critical Points of Strongly-Indefinite Functionals-Part I: General Theory}, J. Funct. Anal. \textbf{162}, 1999, 52--95

\bibitem{Furutani} K. Furutani, \textbf{Fredholm-Lagrangian-Grassmannian and the Maslov index}, J. Geom. Phys. \textbf{51}, 2004, 269--331

\bibitem{Gohberg} I. Gohberg, S. Goldberg, M.A. Kaashoek, \textbf{Classes of linear operators}, Vol. I,
Operator Theory: Advances and Applications \textbf{49}, Birkh\'{a}user Verlag, Basel, 1990

\bibitem{MJN} M. Izydorek, J. Janczewska, N. Waterstraat, \textbf{The Maslov index and the spectral flow - revisited}, Fixed Point Theory Appl. \textbf{2019:5}, 2019

\bibitem{Joachim} M. Joachim, \textbf{Unbounded Fredholm Operators and K-Theory}, Highdimensional Manifold Topology, World Sci. Publishing, 2003, 177-199

\bibitem{Kato} T. Kato, \textbf{Perturbation theory for linear operators}, Reprint of the 1980 edition, Classics in Mathematics, Springer-Verlag, Berlin,  1995

\bibitem{Lesch} M. Lesch, \textbf{The uniqueness of the spectral flow on spaces of unbounded self-adjoint Fredholm operators}, Spectral geometry of manifolds with boundary and decomposition of manifolds, 193--224, Contemp. Math., 366, Amer. Math. Soc., Providence, RI,  2005

\bibitem{Nicolaescu} L. Nicolaescu, \textbf{On the space of Fredholm operators}, An. Stiint. Univ. Al. I. Cuza Iasi. Mat. (N.S.) \textbf{53},   2007, 209--227

\bibitem{Pejsachowicz} J. Pejsachowicz, N. Waterstraat, \textbf{Bifurcation of critical points for continuous families of $C^2$-functionals of Fredholm type}, J. Fixed Point Theory Appl. \textbf{13}, 2013, 537--560

\bibitem{Phillips} J. Phillips, \textbf{Self-adjoint Fredholm operators and spectral flow}, Canad. Math. Bull. \textbf{39}, 1996, 460--467 



\bibitem{Robbin} J. Robbin, D. Salamon, \textbf{The spectral flow and the Maslov index}, Bull. London Math. Soc.  \textbf{27}, 1995, 1--33



\bibitem{Wahl} C. Wahl, \textbf{A new topology on the space of unbounded selfadjoint operators, $K$-theory and spectral flow}, $C^*$-algebras and elliptic theory II, 297--309, Trends. Math., Birkh., Basel, 2008

\bibitem{WaterstraatHomoclinics} N. Waterstraat, \textbf{Spectral flow, crossing forms and homoclinics of Hamiltonian systems}, Proc. Lond. Math. Soc. (3) \textbf{111}, 2015, 275--304

\bibitem{Fredholm} N. Waterstraat, \textbf{Fredholm Operators and Spectral Flow}, Rend. Semin. Mat. Univ. Politec. Torino \textbf{75}, 2017, 7--51

\bibitem{HomoclinicsInf} N. Waterstraat, \textbf{On the Fredholm Lagrangian Grassmannian, Spectral Flow and ODEs in Hilbert Spaces}, arXiv:1803.01143
\end{document}